\documentclass[11pt,a4paper,reqno]{amsart}
\usepackage[T1]{fontenc}
\usepackage{lmodern,amsmath,amssymb,amsthm,mathtools,mathrsfs,microtype}
\usepackage{xcolor}
\usepackage[margin=1.12in]{geometry}
\usepackage[colorlinks=true,linkcolor=blue!50!black,citecolor=blue!50!black,urlcolor=blue!50!black]{hyperref}
\hypersetup{
	pdftitle={Small-time asymptotics of heat kernels of for one-dimensional diffusions in a random environment},
	pdfauthor={Yiduo Wang, Saisai Yang, Tusheng Zhang}
}
\numberwithin{equation}{section}
\allowdisplaybreaks[1]
\newtheorem{theorem}{Theorem}[section]
\newtheorem{proposition}[theorem]{Proposition}
\newtheorem{lemma}[theorem]{Lemma}

\theoremstyle{definition}

\newcommand{\R}{\mathbb R}
\newcommand{\Pbb}{\mathbb P}
\newcommand{\Ebb}{\mathbb E}
\newcommand{\PB}{P_B}

\newcommand{\ann}{\mathrm{ann}}
\newcommand{\dd}{\,\mathrm d}
\newcommand{\cE}{\mathcal E}
\newcommand{\cF}{\mathcal F}
\newcommand{\cL}{\mathscr L}
\DeclareMathOperator{\supp}{supp}
\title[Small-time annealed heat-kernel asymptotics]{Small-time asymptotics of heat kernels of  one-dimensional diffusions in a random environment}
\author{Yiduo Wang}
\thanks{Yiduo Wang: School of Mathematical Sciences, University of Science and Technology of China, Hefei 230026, Anhui, China. Email: \texttt{wangyiduo@mail.ustc.edu.cn}.}
\author{Saisai Yang}
\thanks{Saisai Yang: School of Mathematical Sciences, University of Science and Technology of China, Hefei 230026, Anhui, China. Email: \texttt{yangss@ustc.edu.cn}.}
\author{Tusheng Zhang}
\thanks{Tusheng Zhang: Department of Mathematics, University of Manchester, Manchester M13 9PL, United Kingdom. Email: \texttt{tusheng.zhang@manchester.ac.uk}.}
\date{}
\begin{document}
	\begin{abstract}
		We establish the Varadhan small-time asymptotics for the quenched and annealed
		heat kernels of one-dimensional diffusions in a random environment
		with generator
		\[
		\mathcal L_W f(x)
		=e^{-\rho(x,W)}\bigl(e^{a(x,W)}f'(x)\bigr)'.
		\]
		The coefficients $a$ and $\rho$ are continuous in space and satisfy
		a local exponential moment condition. We assume that the law of
		the intrinsic coordinate map $\Lambda_W$ has compact support
		$\mathscr L$ consisting of strictly increasing functions.
		Let $q^W(t,x,y)$ and $q(t,x,y)=\mathbb E[q^W(t,x,y)]$ denote
		the quenched and annealed heat kernels with respect to Lebesgue
		measure, respectively. We prove that, for almost every environment $W$,
		\[
		\lim_{t\downarrow0}t\log q^W(t,x,y)
		=-\frac12|\Lambda_W(y)-\Lambda_W(x)|^2,
		\]
		and that
		\[
		\lim_{t\downarrow0}t\log q(t,x,y)
		=-\frac12\min_{\Lambda\in\mathscr L}
		|\Lambda(y)-\Lambda(x)|^2.
		\]
		Both limits hold uniformly on compact subsets of $\mathbb R^2$.
		The framework includes Brox diffusion, formally described by
		\[
		dX_t=dB_t-\frac12\dot W(X_t)\,dt,
		\]
		where $B$ is a standard Brownian motion and $W$ is an independent
		two-sided Brownian motion.
	\end{abstract}
	\maketitle
	\section{Introduction}
	Markov processes in random environments have been extensively
	studied since Sinai's seminal work \cite{Sinai} on one-dimensional
	random walks in random media. Subsequent research has addressed
	localization and long-time asymptotics of one-dimensional diffusions
	\cite{Brox,KT,HSY}, as well as homogenization in random media
	\cite{Lejay,Rhodes}.
	
	One important model is Brox diffusion, introduced in \cite{Brox} as
	a continuous-space analogue of Sinai's random walk. Formally, it is
	described by the stochastic differential equation
	\begin{equation}\label{intro:brox}
		\dd X_t=\dd B_t-\frac12\dot W(X_t)\dd t,
	\end{equation}
	where $B$ is a standard Brownian motion and $W$ is an independent
	two-sided Brownian motion representing the environment.
	
	The purpose of this work is to establish the Varadhan small-time asymptotics for
	the quenched and annealed heat kernels of one-dimensional diffusions
	in random environments, including Brox diffusion. More precisely, we
	consider the diffusion $X^W$ associated with the formal generator
	\begin{equation}\label{intro:generator}
		\mathcal L_W f(x)
		=e^{-\rho(x,W)}\bigl(e^{a(x,W)}f'(x)\bigr)',
	\end{equation}
	where $W$ belongs to an environment probability space
	$(\Omega,\mathcal F,\Pbb)$, and $a(\cdot,W)$ and $\rho(\cdot,W)$
	are continuous.
	Let $q^W(t,x,y)$ denote the transition density with respect to Lebesgue
	measure for a fixed environment, and set
	\[
	q(t,x,y)=\Ebb[q^W(t,x,y)],
	\]
	where $\Ebb$ denotes expectation with respect to $\Pbb$.
	Our aim is to identify the limits of $t\log q^W(t,x,y)$ and
	$t\log q(t,x,y)$ as $t\downarrow0$.
	
	One-dimensional diffusions in Brownian potentials have been studied
	from several perspectives. Brox \cite{Brox} established localization
	on the characteristic spatial scale $(\log t)^2$ in the long-time
	regime. For Brownian potentials with an additional linear drift,
	Kawazu and Tanaka \cite{KT} and Hu, Shi and Yor \cite{HSY} studied
	limit theorems and rates of convergence. Hu, L\^e and Mytnik
	\cite{HLM} subsequently gave a rigorous interpretation of the
	singular stochastic differential equation \eqref{intro:brox} and
	proved existence and uniqueness of a strong solution. More recently,
	Chen and Wang \cite{CW} obtained quenched heat-kernel estimates for
	short times and annealed heat-kernel estimates for large times.
	Wang, Yang and Zhang \cite{WYZ} established a small-time annealed
	path large deviation principle for a class of one-dimensional
	diffusions in random environments that includes Brox diffusion.
	
	The study of small-time heat-kernel asymptotics originates with the
	work of Varadhan \cite{Varadhan,Varadhan2}, which relates the leading
	logarithmic behavior of a heat kernel to the squared distance
	associated with the diffusion. Gaussian bounds for divergence-form
	equations were developed by Aronson \cite{Aronson}, and analogous
	upper bounds in the framework of local Dirichlet spaces were obtained
	by Sturm \cite{Sturm}. Norris \cite{Norris} studied heat-kernel
	asymptotics in Lipschitz Riemannian manifolds. For singular drifts,
	Chen, Fang and Zhang \cite{CFZ} established a Varadhan-type formula
	for Brownian motion with measure-valued drift in a suitable Kato class.
	
%
	
	Define the intrinsic coordinate and the associated distance by
	\begin{equation}\label{intro:coordinate}
		\Lambda_W(x)=\int_0^x
		\sqrt{\frac{e^{\rho(z,W)-a(z,W)}}2}\dd z,
		\qquad d_W(x,y)=|\Lambda_W(y)-\Lambda_W(x)|.
	\end{equation}
	Under the continuity and conservativeness hypotheses of Section~2,
	we prove that, for almost every environment $W$,
	\begin{equation}\label{intro:quenched}
		\lim_{t\downarrow0}t\log q^W(t,x,y)
		=-\frac12 d_W(x,y)^2.
	\end{equation}
	For the annealed kernel, we additionally impose a local exponential
	moment condition and assume that
	\[
	\cL=\supp\operatorname{Law}(\Lambda_W)
	\]
	is compact in $C_{\mathrm{loc}}(\R)$ and consists of strictly
	increasing functions. We then obtain
	\begin{equation}\label{intro:annealed}
		\lim_{t\downarrow0}t\log q(t,x,y)
		=-\frac12\min_{\Lambda\in\cL}|\Lambda(y)-\Lambda(x)|^2.
	\end{equation}
	Both limits hold uniformly on compact subsets of $\R^2$.
	For Brox diffusion, $a=-W$ and $\rho=\log2-W$, so
	$\Lambda_W(x)=x$; both limits reduce to $-|x-y|^2/2$.
	
	We establish the lower and upper bounds separately, following the
	probabilistic and analytic approach of \cite{WYZ}. For the lower
	bound, the It\^o--McKean representation and a Cameron--Martin
	estimate give a lower bound for endpoint probabilities. A local
	Gaussian estimate \cite{Aronson} then transfers this bound to the transition
	density, using the short final-time decomposition in \cite{CFZ}.
	The annealed lower bound follows by restricting the expectation to
	a suitable set of environments of positive probability. For the
	upper bound, we use Moser iteration and weighted energy estimates
	\cite{WYZ,Davies} for the one-dimensional Dirichlet form
	\[
	\mathcal E_W(f,g)=\int_\R e^{a(x,W)}f'(x)g'(x)\dd x,
	\qquad \mu_W(\dd x)=e^{\rho(x,W)}\dd x.
	\]
	The dependence of the estimates on the environment is kept explicit
	so that the resulting bound can be averaged under the local moment
	condition. This approach uses continuity of the coefficients and
	does not require their differentiability; the generator and semigroup
	are understood through the Dirichlet form \cite{FOT}.
	
	We now introduce some notation used throughout the article.
	\begin{itemize}
		\item Let $C_{\mathrm{loc}}(\R)$ be the space of all continuous
		real-valued functions on $\R$, endowed with the topology of
		uniform convergence on compact sets.
		
		\item For a compact interval $K$, let $C(K)$ be the space of
		continuous real-valued functions on $K$, equipped with the norm
		\[
		\|f\|_{C(K)}:=\sup_{x\in K}|f(x)|.
		\]
		
		\item The symbols $c$ and $C$ denote generic positive constants
		whose values may change from line to line. Any dependence on
		the environment $W$ is indicated explicitly.
	\end{itemize}
	
	\section{Preliminaries and main results}
	Let $(\Omega,\mathcal F,\Pbb)$ be a probability space, and write
	$\Ebb$ for expectation with respect to $\Pbb$. An element
	$W\in\Omega$ is called an environment. Let
	$a,\rho:\R\times\Omega\to\R$ be jointly measurable functions.
	We consider the one-dimensional diffusion $X^W$ associated with
	the formal generator
	\[
	\mathcal L_W f(x)
	=e^{-\rho(x,W)}\bigl(e^{a(x,W)}f'(x)\bigr)'.
	\]
	For a bounded interval $K$, set
	\begin{equation}\label{eq:M}
		M(K):=\sup_{z\in K}|\rho(z,W)|
		+\sup_{z\in K}|a(z,W)|.
	\end{equation}
	We now list the assumptions on the coefficients.
	
	\medskip
	\noindent
	\textbf{Hypothesis ($H_1$) (Continuity and conservativeness).}
	For $\Pbb$-almost every $W$, the functions $a(\cdot,W)$ and
	$\rho(\cdot,W)$ are continuous. Moreover,
	\begin{equation}\label{eq:H1}
		\int_0^\infty e^{-a(z,W)}\dd z
		=\int_{-\infty}^0 e^{-a(z,W)}\dd z
		=\infty.
	\end{equation}
	
	\medskip
	\noindent
	\textbf{Hypothesis ($H_2$) (Local exponential integrability).}
	For every $R>0$,
	\begin{equation}\label{eq:H2}
		\Ebb\exp\{2M([-R,R])\}<\infty.
	\end{equation}
	
	\medskip
	\noindent
	\textbf{Hypothesis ($H_3$) (Compactness and non-collapse).}
	Define the intrinsic coordinate
	\begin{equation}\label{eq:coordinates}
		\Lambda_W(x):=
		\int_0^x\sqrt{\frac{e^{\rho(z,W)-a(z,W)}}2}\dd z,
		\qquad x\in\R.
	\end{equation}
	We assume that
	\[
	\cL:=\supp\operatorname{Law}(\Lambda_W)
	\]
	is compact in $C_{\mathrm{loc}}(\R)$ and that every
	$\Lambda\in\cL$ is strictly increasing.
	
	\medskip
	Under $H_1$, the diffusion can be constructed by a scale
	transformation and a time change of Brownian motion, as in
	\cite{WYZ}. Let $B=(B(t))_{t\ge0}$ be a standard Brownian motion
	starting from zero, independent of the environment, with law
	$P_B$ and expectation $E_B$. For a fixed environment $W$, define
	\begin{equation}\label{eq:scale-sigma}
		\begin{aligned}
			S_W(x)&:=\int_0^x e^{-a(z,W)}\dd z,\\
			\sigma_W(z)&:=\sqrt2\exp\left\{
			-\frac{\rho(S_W^{-1}(z),W)+a(S_W^{-1}(z),W)}2
			\right\}.
		\end{aligned}
	\end{equation}
	By $H_1$, $S_W$ is an increasing $C^1$ bijection of $\R$
	onto itself. For a fixed starting point $x\in\R$, set
	\[
	\phi_W(t):=
	\int_0^t\sigma_W\bigl(B(s)+S_W(x)\bigr)^{-2}\dd s.
	\]
	The clock $\phi_W$ is continuous and strictly increasing, and
	tends to infinity almost surely. The It\^o--McKean representation
	of the diffusion starting from $x$ is
	\begin{equation}\label{eq:construction}
		X^W(t)=S_W^{-1}\bigl(B(\phi_W^{-1}(t))+S_W(x)\bigr),
		\qquad t\ge0.
	\end{equation}
	
	Denote its quenched law and expectation by $P_x^W$ and $E_x^W$.
	The annealed law is defined by
	\[
	P_x^{\ann}(A):=\int_\Omega P_x^W(A)\,\Pbb(\dd W),
	\qquad
	A\in\mathcal B\bigl(C([0,\infty);\R)\bigr).
	\]
	For a bounded Borel function $f$, the quenched transition
	semigroup is
	\begin{equation}\label{eq:process-semigroup}
		P_t^W f(x):=E_x^W[f(X^W(t))].
	\end{equation}
	
	As proved in \cite{WYZ}, this is the symmetric semigroup
	associated with the regular Dirichlet form
	\begin{equation}\label{eq:form}
		\cE_W(f,g)=\int_\R e^{a(x,W)}f'(x)g'(x)\dd x
	\end{equation}
	on $L^2(\R,\mu_W)$, where
	\[
	\mu_W(\dd x):=e^{\rho(x,W)}\dd x,
	\]
	with domain
	\[
	\cF_W=
	\left\{
	f\in L^2(\mu_W):
	f\text{ is locally absolutely continuous and }
	\int_\R e^{a(x,W)}|f'(x)|^2\dd x<\infty
	\right\}.
	\]
	See also \cite[Chapter~6]{FOT} for the time-change construction
	of symmetric diffusions.
	
	The intrinsic distance associated with $\cE_W$ is
	\begin{equation}\label{eq:distance}
		d_W(x,y):=|\Lambda_W(y)-\Lambda_W(x)|,
		\qquad x,y\in\R.
	\end{equation}
	For $x\in\R$ and $r>0$, write
	\[
	D_W(x,r):=\{y\in\R:d_W(x,y)<r\}.
	\]
	For an interval $K$, let
	\[
	K^r:=\{z\in\R:\operatorname{dist}(z,K)\le r\}
	\]
	denote its closed Euclidean $r$-neighborhood.
	
	The transition density of $X^W$ relative to $\mu_W$ will be denoted
	by $p^W(t,x,y)$, so that
	\[
	P_x^W(X^W(t)\in D)=\int_Dp^W(t,x,y)\mu_W(\dd y),
	\qquad D\in\mathcal B(\R).
	\]
	Its existence is recalled in Lemma~\ref{lem:kernel} in the appendix.
	Define
	\begin{equation}\label{eq:densities}
		q^W(t,x,y):=p^W(t,x,y)e^{\rho(y,W)},
		\qquad
		q(t,x,y):=\Ebb[q^W(t,x,y)].
	\end{equation}
	These are the quenched and annealed transition densities
	with respect to Lebesgue measure, respectively.
	
	We now state our main results.
	
	\begin{theorem}\label{thm:quenched}
		Assume $H_1$. Then, for almost every environment $W$ and
		every compact interval $K\subset\R$,
		\begin{equation}\label{eq:main-quenched}
			\lim_{t\downarrow0}\sup_{x,y\in K}
			\left|
			t\log q^W(t,x,y)+\frac12d_W(x,y)^2
			\right|=0.
		\end{equation}
	\end{theorem}
	
	\begin{theorem}\label{thm:annealed}
		Assume $H_1$--$H_3$, and set
		\begin{equation}\label{eq:J}
			J(x,y)=\frac12\min_{\Lambda\in\cL}
			|\Lambda(y)-\Lambda(x)|^2,
			\qquad x,y\in\R.
		\end{equation}
		Then, for every compact interval $K\subset\R$,
		\begin{equation}\label{eq:main-annealed}
			\lim_{t\downarrow0}\sup_{x,y\in K}
			\left|t\log q(t,x,y)+J(x,y)\right|=0.
		\end{equation}
	\end{theorem}
	
	\section{Lower bound}
	We first establish a local lower bound.
	
	\begin{lemma}\label{lem:coarse}
		For each compact interval $K$ there are $c,C,t_0>0$, depending on
		$a,\rho$ on $K^2$, such that
		\begin{equation}\label{eq:coarse}
			q^W(t,x,y)\ge c t^{-1/2}
			\exp\!\left(-C\frac{|x-y|^2}{t}\right),
			\qquad x,y\in K,\quad0<t\le t_0.
		\end{equation}
	\end{lemma}
	\begin{proof}
		Fix an environment $W$ satisfying $H_1$, and put $M=M(K^2)$.
		All constants below depend only on $M$.
		
		We first establish a lower bound on the diagonal.
		For $x\in K^1$ and $0<r\le1$, let
		\[
		\tau_{x,r}:=\inf\{s\ge0:|X^W(s)-x|\ge r\}.
		\]
		Since
		\[
		d_W(x,x\pm r)\ge 2^{-1/2}e^{-M/2}r,
		\]
		the exit estimate obtained at the end of the proof of
		\cite[Theorem~2.2]{WYZ} gives
		\[
		P_x^W(\tau_{x,r}\le h)
		\le Ce^{CM}\exp\left(-c e^{-M}\frac{r^2}{h}\right)
		\le Ce^{CM}\frac{h}{r^2},
		\qquad 0<h<1.
		\]
		Choose $R\ge1$ sufficiently large and then $t_0\in(0,1)$
		such that $R\sqrt{t_0}\le1$. Applying this estimate with
		$r=R\sqrt t$ and $h=t/2$, we obtain
		\[
		P_x^W\bigl(|X^W(t/2)-x|<R\sqrt t\bigr)\ge\frac12,
		\qquad x\in K^1,\quad 0<t\le t_0.
		\]
		By symmetry, the semigroup identity and Cauchy--Schwarz,
		\[
		\begin{aligned}
			p^W(t,x,x)
			&=\int_\R p^W(t/2,x,z)^2\,\mu_W(\dd z)\\
			&\ge
			\frac{
				\left(\displaystyle\int_{x-R\sqrt t}^{x+R\sqrt t}
				p^W(t/2,x,z)\,\mu_W(\dd z)\right)^2
			}{
				\mu_W([x-R\sqrt t,x+R\sqrt t])
			}\\
			&\ge \frac{1}{8e^MR\sqrt t}
			=:c_0t^{-1/2}.
		\end{aligned}
		\]
		
		By \eqref{eq:row-bound} and \eqref{1}, we obtain
		\[
		\begin{aligned}
			&\cE_W\bigl(p^W(t,x,\cdot),p^W(t,x,\cdot)\bigr)^{1/2}\\
			&\qquad\le
			t^{-1/2}\|p^W(t/2,x,\cdot)\|_{L^2(\mu_W)}
			\le Ct^{-3/4},
			\qquad x\in K^1,\quad 0<t\le t_0.
		\end{aligned}
		\]
		For $x,y\in K^1$, the fundamental theorem of calculus and
		Cauchy--Schwarz yield
		\[
		\begin{aligned}
			|p^W(t,x,y)-p^W(t,x,x)|
			&\le
			\left(\int_{x\wedge y}^{x\vee y}e^{-a(z,W)}\,\dd z\right)^{1/2}\\
			&\qquad\times
			\cE_W\bigl(p^W(t,x,\cdot),p^W(t,x,\cdot)\bigr)^{1/2}\\
			&\le Ct^{-3/4}|x-y|^{1/2}.
		\end{aligned}
		\]
		Choose $\delta\in(0,1]$ such that $C\sqrt\delta\le c_0/2$.
		Combining this estimate with the diagonal lower bound gives
		\[
		p^W(t,x,y)\ge c_1t^{-1/2},
		\qquad
		x,y\in K^1,\quad |x-y|\le\delta\sqrt t,\quad 0<t\le t_0,
		\]
		where $c_1=c_0/2$.
		
		Finally, fix $x,y\in K$ and $0<t\le t_0$, and set
		\[
		n=\max\left\{1,
		\left\lceil\frac{4|x-y|^2}{\delta^2t}\right\rceil\right\},
		\qquad h=\frac tn,\qquad
		x_i=x+\frac{i}{n}(y-x),\quad 0\le i\le n.
		\]
		If $n=1$, the desired lower bound follows from the preceding
		estimate. Suppose $n\ge2$, and define
		\[
		I_i=\left[x_i-\frac{\delta\sqrt h}{8},
		x_i+\frac{\delta\sqrt h}{8}\right],
		\qquad 1\le i\le n-1.
		\]
		These intervals lie in $K^1$. For $z_0=x$, $z_n=y$ and
		$z_i\in I_i$, we have
		\[
		|z_i-z_{i-1}|
		\le\frac{|x-y|}{n}+\frac{\delta\sqrt h}{4}
		\le\delta\sqrt h.
		\]
		Moreover,
		\[
		\mu_W(I_i)\ge\frac{\delta e^{-M}}4\sqrt h.
		\]
		Repeated application of the semigroup identity therefore gives
		\[
		\begin{aligned}
			p^W(t,x,y)
			&\ge
			\int_{I_1\times\cdots\times I_{n-1}}
			\prod_{i=1}^{n}p^W(h,z_{i-1},z_i)
			\prod_{i=1}^{n-1}\mu_W(\dd z_i)\\
			&\ge
			(c_1h^{-1/2})^n
			\left(\frac{\delta e^{-M}}4\sqrt h\right)^{n-1}\\
			&\ge ct^{-1/2}e^{-Cn}\\
			&\ge ct^{-1/2}
			\exp\left(-C\frac{|x-y|^2}{t}\right),
		\end{aligned}
		\]
		where the last step uses
		$n\le1+4|x-y|^2/(\delta^2t)$.
		Multiplying by $e^{\rho(y,W)}\ge e^{-M}$ proves
		\eqref{eq:coarse}.
	\end{proof}
	\begin{lemma}
		\label{lem:endpoint}
		For a fixed environment, a compact interval $K$, $\tau>0$ and $r>0$,
		\begin{equation}\label{eq:endpoint}
			\liminf_{t\downarrow0}\inf_{x,y\in K}
			\left[t\log P_x^W(|X^W(\tau t)-y|<r)
			+\frac{d_W(x,y)^2}{2\tau}\right]\ge0.
		\end{equation}
	\end{lemma}
	
	\begin{proof}
		Fix $W,K,\tau$ and $r$.
		For $x\ne y$, put
		\[
		v=\frac{\Lambda_W(y)-\Lambda_W(x)}{\tau},\qquad
		T=\frac1v\int_{S_W(x)}^{S_W(y)}\sigma(z)\dd z,
		\]
		and define $h$ on $[0,T]$ by
		\begin{equation}\label{eq:control-definition}
			\int_{S_W(x)}^{h(u)}\sigma(z)\dd z=vu.
		\end{equation}
		The integral in \eqref{eq:control-definition}
		is strictly increasing in its upper endpoint. Thus $h$ is well
		defined, with
		\[
		h(0)=S_W(x),\qquad h(T)=S_W(y),\qquad
		h'(u)=\frac{v}{\sigma(h(u))}.
		\]
		Changing variables gives
		\begin{equation}\label{eq:control-clock}
			\int_0^T\sigma(h(u))^{-2}\dd u
			=\frac1v\int_{S_W(x)}^{S_W(y)}\sigma(z)^{-1}\dd z
			=\frac{\Lambda_W(y)-\Lambda_W(x)}v=\tau.
		\end{equation}
		When $x=y$, take $h(u)=S_W(x)$ and
		$T=\tau\sigma(S_W(x))^2$; then \eqref{eq:control-clock} also holds.
		In both cases, extend $h$ constantly after $T$. Its range lies in
		$S_W(K)$, and
		\begin{equation}\label{eq:control-energy}
			\frac12\int_0^\infty|h'(u)|^2\dd u
			=\frac{v^2}{2}\int_0^T\sigma(h(u))^{-2}\dd u
			=\frac{d_W(x,y)^2}{2\tau},
		\end{equation}
		where $v=0$ in the case $x=y$.
		
		We next choose a common time horizon and tube radius for all
		$x,y\in K$. Put
		\[
		J=S_W(K^1),\qquad m=\min_{z\in J}\sigma(z)^{-2}>0,
		\qquad T_*=\frac{\tau}{m}.
		\]
		By \eqref{eq:control-clock}, $T\le T_*$. Moreover, the paths $h$
		have a common Lipschitz constant
		\[
		L=\frac{\max_{x,y\in K}d_W(x,y)}{\tau}
		\max_{z\in S_W(K)}\sigma(z)^{-1}.
		\]
		Let
		\[
		\omega(\delta)=
		\sup_{\substack{z,z'\in J\\|z-z'|\le\delta}}
		\bigl|\sigma(z)^{-2}-\sigma(z')^{-2}\bigr|.
		\]
		Then $\omega(\delta)\to0$ as $\delta\downarrow0$.
		Also $S_W^{-1}$ is Lipschitz on $J$; denote a Lipschitz constant
		by $C_S$. Choose $\delta>0$ such that
		\begin{equation}\label{eq:tube-radius}
			\delta<\operatorname{dist}(S_W(K),\partial J),\qquad
			C_S\left(\delta+\frac{LT_*}{m}\omega(\delta)\right)<r.
		\end{equation}
		This choice depends only on $W,K,\tau,r$.
		
		For $t>0$, consider the event
		\begin{equation}\label{eq:scale-tube}
			E_t(h,\delta)=
			\left\{\sup_{0\le u\le T_*}
			|S_W(x)+\widetilde B_{tu}-h(u)|<\delta\right\}.
		\end{equation}
		On this event, $B_{tu}\in J$ for $u\le T_*$, and hence
		\[
		\frac{\phi(tT_*)}{t}
		=\int_0^{T_*}\sigma(B_{tu})^{-2}\dd u
		\ge mT_*=\tau.
		\]
		Consequently $U:=\phi^{-1}(\tau t)/t\le T_*$, and
		\[
		\int_0^U\sigma(B_{tu})^{-2}\dd u
		=\tau=\int_0^T\sigma(h(u))^{-2}\dd u.
		\]
		Subtracting these identities yields
		\begin{align*}
			m|U-T|
			&\le\left|\int_T^U\sigma(h(u))^{-2}\dd u\right|\\
			&=\left|\int_0^U
			\bigl(\sigma(h(u))^{-2}-\sigma(B_{tu})^{-2}\bigr)\dd u\right|
			\le T_*\omega(\delta).
		\end{align*}
		Using $h(T)=S_W(y)$ and the Lipschitz bound for $h$, we obtain
		\[
		|B_{tU}-S_W(y)|
		\le |B_{tU}-h(U)|+|h(U)-h(T)|
		<\delta+\frac{LT_*}{m}\omega(\delta).
		\]
		It follows from \eqref{eq:tube-radius} that
		\begin{equation}\label{eq:event-inclusion}
			E_t(h,\delta)\subset
			\bigl\{|X^W(\tau t)-y|<r\bigr\},
			\qquad x,y\in K,\quad t>0.
		\end{equation}
		
		Finally, the Cameron--Martin formula and Brownian scaling give
		\begin{equation}\label{eq:CM-tube}
			\PB(E_t(h,\delta))
			\ge \exp\!\left(-\frac1{2t}
			\int_0^{T_*}|h'(u)|^2\dd u\right)b_t,
			\qquad
			b_t=\PB\!\left(\sup_{u\le T_*}|\widetilde B_{tu}|<\delta\right).
		\end{equation}
		
		Moreover,
		\[
		b_t=\PB\!\left(\sqrt t\sup_{u\le T_*}|\widetilde B_u|<\delta\right)
		\longrightarrow1.
		\]
		Combining \eqref{eq:control-energy}, \eqref{eq:event-inclusion}
		and \eqref{eq:CM-tube}, we conclude that
		\begin{equation}\label{eq:endpoint-probability-lower}
			P_x^W\bigl(|X^W(\tau t)-y|<r\bigr)
			\ge b_t\exp\!\left(-\frac{d_W(x,y)^2}{2\tau t}\right),
			\qquad x,y\in K.
		\end{equation}
		Since $b_t$ is independent of $x,y$ and $t\log b_t\to0$,
		taking logarithms and then the infimum over $x,y\in K$ proves
		\eqref{eq:endpoint}.
	\end{proof}
	
	\begin{proposition}\label{prop:quenched-lower}
		For every compact interval $K$,
		\begin{equation}\label{eq:quenched-lower}
			\liminf_{t\downarrow0}\inf_{x,y\in K}
			\left(t\log q^W(t,x,y)+\frac{d_W(x,y)^2}{2}\right)\ge0.
		\end{equation}
	\end{proposition}
	
	\begin{proof}
		Fix $0<\varepsilon<1$ and $0<r<1$, independently of $t$.
		The quenched semigroup identity and positivity imply
		\begin{align*}
			q^W(t,x,y)
			&\ge\int_{|z-y|<r}q^W((1-\varepsilon)t,x,z)
			q^W(\varepsilon t,z,y)\dd z\\
			&\ge P_x^W(|X^W((1-\varepsilon)t)-y|<r)
			\inf_{|z-y|<r}q^W(\varepsilon t,z,y).
		\end{align*}
		The last infimum is taken over $z\in K^1$ when $y\in K$.
		Apply Lemma~\ref{lem:coarse} on the fixed interval $K^1$. Its
		constants $c,C,t_0$ can thus be chosen independently of
		$r,\varepsilon,x,y$. For $\varepsilon t\le t_0$,
		\[
		\inf_{|z-y|<r}q^W(\varepsilon t,z,y)
		\ge c(\varepsilon t)^{-1/2}
		\exp\!\left(-\frac{Cr^2}{\varepsilon t}\right).
		\]
		For the first factor use \eqref{eq:endpoint-probability-lower} with
		$\tau=1-\varepsilon$. It gives a number $b_t>0$, depending on
		$W,K,r,\varepsilon$ but not on $x,y$, such that $b_t\to1$ and
		\[
		q^W(t,x,y)\ge c(\varepsilon t)^{-1/2}b_t
		\exp\!\left(-\frac{d_W(x,y)^2}{2(1-\varepsilon)t}
		-\frac{Cr^2}{\varepsilon t}\right).
		\]
		Then
		\begin{align*}
			t\log q^W(t,x,y)+\frac{d_W(x,y)^2}{2}
			\ge{}&t\log\!\big(c(\varepsilon t)^{-1/2}b_t\big)
			-\frac{Cr^2}{\varepsilon}\\
			&-\frac{\varepsilon d_W(x,y)^2}{2(1-\varepsilon)}.
		\end{align*}
		The logarithmic term tends to zero for fixed $r,\varepsilon$.
		The function $d_W$ is bounded on $K\times K$. Therefore
		\begin{align*}
			\liminf_{t\downarrow0}\inf_{x,y\in K}
			\left(t\log q^W(t,x,y)+\frac{d_W(x,y)^2}{2}\right)
			\ge -\frac{Cr^2}{\varepsilon}
			-\frac{\varepsilon\sup_{x,y\in K}d_W(x,y)^2}{2(1-\varepsilon)}.
		\end{align*}
		First let $r\downarrow0$ and then let $\varepsilon\downarrow0$
		to obtain \eqref{eq:quenched-lower}.
	\end{proof}
	
	\begin{proposition}\label{prop:annealed-lower}
		Under $H_1$--$H_3$, for every compact interval $K$,
		\begin{equation}\label{eq:annealed-lower}
			\liminf_{t\downarrow0}\inf_{x,y\in K}
			\bigl(t\log q(t,x,y)+J(x,y)\bigr)\ge0.
		\end{equation}
	\end{proposition}
	
	\begin{proof}
		By $H_3$, the minimum in the definition of $J$ is attained.
		Moreover, $J$ is continuous and satisfies $J(x,y)>0$ whenever $x\ne y$.
		
		Fix $\delta>0$. For each $(x_0,y_0)\in K\times K$, choose
		$\Lambda_0\in\cL$ such that
		\[
		\frac12d_{\Lambda_0}(x_0,y_0)^2=J(x_0,y_0).
		\]
		By continuity, there exist a neighborhood $V$ of $\Lambda_0$ with positive probability
		in $\cL$ and a relative open neighborhood $U$ of $(x_0,y_0)$
		in $K\times K$ such that
		\[
		\frac12d_\Lambda(x,y)^2\le J(x,y)+\delta,
		\qquad \Lambda\in V,\quad (x,y)\in U.
		\]
		
		These neighborhoods $U$ cover $K\times K$. By compactness,
		we can select a finite subcover $U_1,\ldots,U_m$.
		The corresponding events $A_1,\ldots,A_m$ all have positive
		probability and satisfy
		\begin{equation}\label{eq:favorable}
			\frac12d_W(x,y)^2\le J(x,y)+\delta,
			\qquad W\in A_i,\quad (x,y)\in U_i.
		\end{equation}
		
		For $i=1,\ldots,m$ and $n\ge1$, set
		\[
		B_{i,n}
		=
		A_i\cap
		\left\{
		W:
		t\log q^W(t,x,y)+\frac12d_W(x,y)^2\ge-\delta
		\ \text{for all }(x,y)\in K\times K,\ 0<t\le n^{-1}
		\right\}.
		\]
		These events are measurable.
		Proposition~\ref{prop:quenched-lower} implies that
		$B_{i,n}\uparrow A_i$ up to a null set. Hence, for each $i$,
		we may choose $n_i$ such that
		$\Pbb(B_{i,n_i})>0$.
		
		For $(x,y)\in U_i$ and $0<t\le n_i^{-1}$,
		the definition of $B_{i,n_i}$ and \eqref{eq:favorable} give
		\[
		\begin{aligned}
			q(t,x,y)
			&\ge \Ebb\!\left[
			\mathbf1_{B_{i,n_i}}q^W(t,x,y)
			\right]\\
			&\ge \Pbb(B_{i,n_i})
			\exp\!\left(-\frac{J(x,y)+2\delta}{t}\right).
		\end{aligned}
		\]
		Set
		\[
		c_\delta=\min_{1\le i\le m}\Pbb(B_{i,n_i})>0,
		\qquad
		t_\delta=\min_{1\le i\le m}n_i^{-1}>0.
		\]
		Since the sets $U_i$ cover $K\times K$, we obtain
		\[
		q(t,x,y)\ge c_\delta
		\exp\!\left(-\frac{J(x,y)+2\delta}{t}\right),
		\qquad x,y\in K,\quad 0<t\le t_\delta.
		\]
		Consequently,
		\[
		\liminf_{t\downarrow0}\inf_{x,y\in K}
		\bigl(t\log q(t,x,y)+J(x,y)\bigr)\ge-2\delta.
		\]
		Letting $\delta\downarrow0$ proves the assertion.
	\end{proof}
	\section{Annealed upper bound}\label{sec:upper}
	We use the Caccioppoli inequality and local iteration in
	\cite[Lemmas~4.3--4.6]{WYZ}. Fix an environment satisfying $H_1$,
	and write $P_t=P_t^W$, $\cE(u)=\cE_W(u,u)$ and
	$\|u\|_2=\|u\|_{L^2(\mu_W)}$.
	
	For a bounded, real, locally absolutely continuous function $\psi$
	such that
	\begin{equation}\label{eq:intrinsic-weight}
		2e^{a-\rho}|\psi'|^2\le h^2\quad\text{almost everywhere},
		\qquad h\ge0,
	\end{equation}
	put $P_s^\psi f=e^\psi P_s(e^{-\psi}f)$. The weighted energy
	estimate \cite[Chapter~3]{Davies} gives
	\begin{equation}\label{eq:L2}
		\|P_s^\psi f\|_2\le e^{h^2s/2}\|f\|_2.
	\end{equation}
	
	\begin{lemma}\label{lem:weighted-point}
		Let $K$ be a compact interval and $M=M(K^1)$. For $\psi,h$ as in
		\eqref{eq:intrinsic-weight}, $f\in L^2(\mu_W)$, $x\in K$ and
		$0<s\le1$,
		\begin{equation}\label{eq:point-smooth}
			|P_s^\psi f(x)|\le Cs^{-1/4}e^{M/2}e^{h^2s}\|f\|_2,
		\end{equation}
		where $C$ is independent of $W,K,h$ and $s$.
	\end{lemma}
	\begin{proof}
		Write $M=M(K^1)$ and choose
		\[
		\ell=\frac14e^{-M/2}\sqrt{s},\qquad
		I=D_W(x,\ell),\qquad
		Q=(s-\ell^2,s+\ell^2)\times I.
		\]
		Then $Q\subset(0,\infty)\times K^1$.
		
		We use the local mean-value argument in
		\cite[Proof of Lemma~4.6]{WYZ}.
		
		The same iteration gives
		\[
		u(s,x)^2
		\le Ce^{M/2}\ell^{-3}
		\int_{s-\ell^2}^{s+\ell^2}\int_I
		u(r,z)^2\,\mu_W(\dd z)\dd r.
		\]
		Here the exponent $M/2$ follows from
		$\sum_{n\ge0}2/(3p_n)=1/2$, where $p_n=2\cdot3^n$.
		
		Apply this estimate to
		$u(r,z)=P_r(e^{-\psi}|f|)(z)$.
		Since $|\psi(z)-\psi(x)|\le h\ell$ on $I$,
		positivity and \eqref{eq:L2} yield
		\begin{align*}
			|P_s^\psi f(x)|^2
			&\le Ce^{M/2}\ell^{-3}e^{2h\ell}
			\int_{s-\ell^2}^{s+\ell^2}
			\|P_r^\psi|f|\|_2^2\,\dd r\\
			&\le Ce^{M/2}\ell^{-1}
			e^{2h\ell+h^2(s+\ell^2)}\|f\|_2^2.
		\end{align*}
		Taking square roots gives \eqref{eq:point-smooth}, since
		\[
		e^{M/4}\ell^{-1/2}=2e^{M/2}s^{-1/4},
		\qquad
		h\ell+\frac12h^2(s+\ell^2)\le h^2s+1.
		\qedhere
		\]
	\end{proof}
	\begin{proposition}\label{prop:upper}
		For each compact interval $K$ and $\eta\in(0,1)$,
		\begin{equation}\label{eq:quenched-upper}
			q^W(t,x,y)\le C_\eta t^{-1/2}e^{2M(K^1)}
			\exp\!\left(-\frac{d_W(x,y)^2}{2(1+\eta)t}\right),
			\quad x,y\in K,\quad0<t\le1,
		\end{equation}
		where $C_\eta$ is independent of $W$ and $K$.
	\end{proposition}
	\begin{proof}
		Let $s=\eta t/2$ and write $M=M(K^1)$. For fixed $y\in K$, set
		\[
		g_y(z)=e^{\psi(z)-\psi(y)}p^W(s,z,y).
		\]
		By symmetry of $p^W$,
		\[
		\int_\R g_y(z)f(z)\,\mu_W(\dd z)=P_s^{-\psi}f(y).
		\]
		Thus, applying \eqref{eq:point-smooth} with $-\psi$ and using
		$L^2$ duality, we obtain
		\[
		\|g_y\|_2\le Cs^{-1/4}e^{M/2}e^{h^2s}.
		\]
		
		The semigroup identity gives
		\[
		e^{\psi(x)-\psi(y)}p^W(t,x,y)
		=\bigl(P_s^\psi P_{t-2s}^\psi g_y\bigr)(x).
		\]
		Apply \eqref{eq:point-smooth} at $x$, and then use
		\eqref{eq:L2} to estimate $\|P_{t-2s}^\psi g_y\|_2$.
		Together with the bound for $\|g_y\|_2$, this yields
		\begin{align*}
			e^{\psi(x)-\psi(y)}p^W(t,x,y)
			&\le Cs^{-1/4}e^{M/2}e^{h^2s}
			\|P_{t-2s}^\psi g_y\|_2\\
			&\le Cs^{-1/4}e^{M/2}e^{h^2s}
			e^{h^2(t-2s)/2}\|g_y\|_2\\
			&\le Cs^{-1/2}e^M
			\exp\!\left(2h^2s+\frac{h^2(t-2s)}2\right)\\
			&=C_\eta t^{-1/2}e^{M(K^1)}
			\exp\!\left(\frac{(1+\eta)h^2t}{2}\right),
		\end{align*}
		where the last equality follows from $s=\eta t/2$.
		For $h\ge0$, choose the bounded weight
		\[
		\psi(z)=h\min\{d_W(z,y),d_W(x,y)\}.
		\]
		It satisfies \eqref{eq:intrinsic-weight} and
		$\psi(x)-\psi(y)=h\,d_W(x,y)$. Multiplying the preceding
		bound by $e^{\rho(y,W)}\le e^{M(K^1)}$ yields
		\[
		q^W(t,x,y)\le C_\eta t^{-1/2}e^{2M(K^1)}
		\exp\!\left(\frac{(1+\eta)h^2t}{2}-h\,d_W(x,y)\right).
		\]
		Taking $h=d_W(x,y)/((1+\eta)t)$ proves the assertion.
	\end{proof}
	
	\begin{proof}[Proof of Theorems~\ref{thm:quenched} and~\ref{thm:annealed}]
		Fix a compact interval $K$.
		
		We first prove Theorem~\ref{thm:quenched}. Fix an environment $W$
		satisfying $H_1$. For each $\eta\in(0,1)$,
		\eqref{eq:quenched-upper} gives
		\[
		t\log q^W(t,x,y)+\frac12d_W(x,y)^2
		\le
		t\log\!\left(C_\eta t^{-1/2}e^{2M(K^1)}\right)
		+\frac{\eta}{2(1+\eta)}d_W(x,y)^2.
		\]
		For fixed $W$ and $\eta$, the first term on the right tends to
		zero as $t\downarrow0$. Since $d_W$ is bounded on $K\times K$,
		we obtain
		\[
		\limsup_{t\downarrow0}\sup_{x,y\in K}
		\left(t\log q^W(t,x,y)+\frac12d_W(x,y)^2\right)
		\le
		\frac{\eta}{2(1+\eta)}
		\sup_{x,y\in K}d_W(x,y)^2.
		\]
		Letting $\eta\downarrow0$ and combining this with
		Proposition~\ref{prop:quenched-lower} proves
		\eqref{intro:quenched}, uniformly on $K\times K$.
		
		We next prove Theorem~\ref{thm:annealed}.
		Finiteness, positivity and continuity of $q$ follow from
		Lemmas~\ref{lem:kernel} and~\ref{lem:coarse}.
		Almost surely, $\Lambda_W\in\cL$, and hence
		\[
		\frac12d_W(x,y)^2\ge J(x,y),
		\qquad x,y\in\R.
		\]
		Taking expectation in \eqref{eq:quenched-upper} therefore yields
		\begin{equation}\label{eq:annealed-upper}
			q(t,x,y)\le
			C_\eta t^{-1/2}\Ebb\!\left[e^{2M(K^1)}\right]
			\exp\!\left(-\frac{J(x,y)}{(1+\eta)t}\right),
			\qquad x,y\in K,\quad 0<t\le1.
		\end{equation}
		The expectation is finite by $H_2$. Taking logarithms and
		multiplying by $t$, we obtain
		\[
		t\log q(t,x,y)+J(x,y)
		\le
		t\log\!\left(
		C_\eta t^{-1/2}\Ebb\!\left[e^{2M(K^1)}\right]
		\right)
		+\frac{\eta}{1+\eta}J(x,y).
		\]
		The first term on the right tends to zero as $t\downarrow0$.
		Since $J$ is bounded on $K\times K$, it follows that
		\[
		\limsup_{t\downarrow0}\sup_{x,y\in K}
		\bigl(t\log q(t,x,y)+J(x,y)\bigr)
		\le
		\frac{\eta}{1+\eta}\sup_{x,y\in K}J(x,y).
		\]
		Letting $\eta\downarrow0$ and combining this with
		Proposition~\ref{prop:annealed-lower} proves
		\eqref{intro:annealed}, uniformly on $K\times K$.
	\end{proof}
	\section{Appendix}
	\begin{lemma}\label{lem:kernel}
		Under $H_1$, the transition density $p^W$ has a symmetric,
		nonnegative, jointly continuous version on $(0,\infty)\times\R^2$,
		which we use throughout. The density $q^W$ is jointly measurable
		in $(W,t,x,y)$. For $s>0$ and $f\in L^2(\mu_W)$,
		\begin{equation}\label{eq:unweighted-smooth}
			|P_sf(x)|\le C s^{-1/4}
			e^{M([x-\sqrt s,x+\sqrt s])/2}\|f\|_2.
		\end{equation}
		Under $H_2$, $q$ is finite and jointly continuous for positive times.
	\end{lemma}
	\begin{proof}
		Fix an environment $W$ satisfying $H_1$. The existence,
		symmetry, and joint continuity of $p^W$ follow from the
		classical theory of regular one-dimensional diffusions; see
		\cite[Section~2.1, equations~(1)--(2)]{PitmanYor}.
		Consequently, $q^W(t,x,y)=p^W(t,x,y)e^{\rho(y,W)}$ is jointly
		continuous in $(t,x,y)$.
		
		The scale--time construction is jointly measurable in the
		environment, time, starting point, and Brownian path.
		Therefore, the identity
		\[
		q^W(t,x,y)
		=\lim_{n\to\infty}\frac n2
		P_x^W\bigl(|X^W(t)-y|<1/n\bigr)
		\]
		shows that $q^W$ is jointly measurable in $(W,t,x,y)$.
		
		It remains to prove \eqref{eq:unweighted-smooth}.
		Fix $s>0$ and let $f\in L^2(\mu_W)$. Write $\|\cdot\|_2$
		for the norm in $L^2(\mu_W)$. By the standard semigroup
		estimates \cite[Section~1.3]{FOT}, $P_s^Wf\in\cF_W$ and
		
		\begin{equation}\label{1}
			\|P_s^W f\|_2 \le \|f\|_2,
			\qquad
			\mathcal{E}_W(P_s^W f,P_s^W f)^{1/2}
			\le (2s)^{-1/2}\|f\|_2.
		\end{equation}
		
		In particular, $P_s^Wf$ is locally absolutely continuous.
		
		Let $u=P_s^Wf$. Fix $x\in\R$ and $r>0$, and put
		$I=[x-r,x+r]$. For every $z\in I$, the fundamental theorem
		of calculus gives
		\[
		|u(x)|^2
		\le |u(z)|^2+2\int_I|u(y)u'(y)|\dd y.
		\]
		Integrating this inequality with respect to $z$ over $I$
		and dividing by $2r$, we obtain
		\[
		|u(x)|^2
		\le \frac{1}{2r}\int_I|u(z)|^2\dd z
		+2\int_I|u(z)u'(z)|\dd z.
		\]
		
		To bound the first integral, use
		$\mu_W(\dd z)=e^{\rho(z,W)}\dd z$:
		\[
		\int_I|u(z)|^2\dd z
		\le e^{\sup_I|\rho(\cdot,W)|}\|u\|_2^2
		\le e^{M(I)}\|u\|_2^2.
		\]
		For the second integral, Cauchy--Schwarz gives
		\[
		\begin{aligned}
			\int_I|u(z)u'(z)|\dd z
			&=\int_I e^{-(a(z,W)+\rho(z,W))/2}
			|e^{\rho(z,W)/2}u(z)|
			|e^{a(z,W)/2}u'(z)|\dd z\\
			&\le e^{M(I)/2}
			\left(\int_I e^{\rho(z,W)}|u(z)|^2\dd z\right)^{1/2}
			\left(\int_I e^{a(z,W)}|u'(z)|^2\dd z\right)^{1/2}\\
			&\le e^{M(I)/2}\|u\|_2\cE_W(u,u)^{1/2}.
		\end{aligned}
		\]
		Combining these estimates and using $M(I)\ge0$ yields
		\[
		|P_s^Wf(x)|^2
		\le e^{M(I)}
		\left(
		\frac{\|u\|_2^2}{2r}
		+2\|u\|_2\cE_W(u,u)^{1/2}
		\right).
		\]
		The semigroup estimates stated above now imply
		\[
		|P_s^Wf(x)|^2
		\le Ce^{M(I)}
		\bigl(r^{-1}+s^{-1/2}\bigr)\|f\|_2^2.
		\]
		Taking $r=\sqrt{s}$ and then taking square roots gives
		\[
		|P_s^Wf(x)|
		\le Cs^{-1/4}
		e^{M([x-\sqrt{s},x+\sqrt{s}])/2}\|f\|_2,
		\]
		which proves \eqref{eq:unweighted-smooth}.
		
		Consequently, for fixed $s>0$ and
		$x\in\R$,
		\begin{equation}\label{eq:row-bound}
			\begin{aligned}
				\|p^W(s,x,\cdot)\|_2
				&=\sup_{\substack{f\in L^2(\mu_W)\\ \|f\|_2\le1}}
				|P_s^Wf(x)|\\
				&\le Cs^{-1/4}
				e^{M([x-\sqrt{s},x+\sqrt{s}])/2}.
			\end{aligned}
		\end{equation}
		
		Fix $0<u<v<\infty$ and a compact interval $K$, and choose
		$0<s_0<\min(u/2,1)$. By symmetry, the semigroup property,
		and $L^2$ contraction,
		\[
		\begin{aligned}
			q^W(t,x,y)
			&\le e^{\rho(y,W)}
			\|p^W(s_0,x,\cdot)\|_2
			\|p^W(s_0,y,\cdot)\|_2\\
			&\le C_u e^{2M(K^1)},
			\qquad t\in[u,v],\quad x,y\in K.
		\end{aligned}
		\]
		Under $H_2$, the last bound is integrable in $W$.
		Dominated convergence therefore proves that
		$q=\Ebb[q^W]$ is finite and jointly continuous for positive
		times.
	\end{proof}


\begin{thebibliography}{99}
		
		\bibitem{Aronson}
		D.~G. Aronson,
		Non-negative solutions of linear parabolic equations,
		\emph{Ann. Scuola Norm. Sup. Pisa Cl. Sci.} (3) \textbf{22} (1968),
		no.~4, 607--694.
		\href{https://www.numdam.org/item/ASNSP_1968_3_22_4_607_0/}{Full text}.
		
		\bibitem{Brox}
		T. Brox,
		A one-dimensional diffusion process in a Wiener medium,
		\emph{Ann. Probab.} \textbf{14} (1986), no.~4, 1206--1218.
		
		\bibitem{CW}
		X. Chen and J. Wang,
		Quenched and annealed heat kernel estimates for Brox's diffusion,
		\emph{Probab. Theory Related Fields} (2026).
		\href{https://doi.org/10.1007/s00440-026-01518-5}{doi:10.1007/s00440-026-01518-5}.
		
		\bibitem{CFZ}
		Z.-Q. Chen, S. Fang and T. Zhang,
		Small time asymptotics for Brownian motion with singular drift,
		\emph{Proc. Amer. Math. Soc.} \textbf{147} (2019), no.~8, 3567--3578.
		\href{https://arxiv.org/abs/1808.02326}{arXiv:1808.02326}.
		
		\bibitem{Davies}
		E.~B. Davies,
		\emph{Heat Kernels and Spectral Theory},
		Cambridge Tracts in Mathematics 92,
		Cambridge University Press, Cambridge, 1989.
		
		\bibitem{FOT}
		M. Fukushima, Y. Oshima and M. Takeda,
		\emph{Dirichlet Forms and Symmetric Markov Processes},
		2nd revised and extended ed., De Gruyter, Berlin, 2011.
		
		\bibitem{HLM}
		Y. Hu, K. L\^e and L. Mytnik,
		Stochastic differential equation for Brox diffusion,
		\emph{Stochastic Process. Appl.} \textbf{127} (2017), no.~7, 2281--2315.
		\href{https://arxiv.org/abs/1506.02280}{arXiv:1506.02280}.
		
		\bibitem{HSY}
		Y. Hu, Z. Shi and M. Yor,
		Rates of convergence of diffusions with drifted Brownian potentials,
		\emph{Trans. Amer. Math. Soc.} \textbf{351} (1999), no.~10, 3915--3934.
		
		\bibitem{KT}
		K. Kawazu and H. Tanaka,
		A diffusion process in a Brownian environment with drift,
		\emph{J. Math. Soc. Japan} \textbf{49} (1997), no.~2, 189--211.
		
		\bibitem{Lejay}
		A. Lejay,
		Homogenization of divergence-form operators with lower-order terms
		in random media,
		\emph{Probab. Theory Related Fields} \textbf{120} (2001), no.~2,
		255--276.
		
		\bibitem{Norris}
		J.~R. Norris,
		Heat kernel asymptotics and the distance function in Lipschitz
		Riemannian manifolds,
		\emph{Acta Math.} \textbf{179} (1997), no.~1, 79--103.
		
		\bibitem{Rhodes}
		R. Rhodes,
		Diffusion in a locally stationary random environment,
		\emph{Probab. Theory Related Fields} \textbf{143} (2009), no.~3--4,
		545--568.
		
		\bibitem{Sinai}
		Ya.~G. Sinai,
		The limit behavior of a one-dimensional random walk in a random
		environment,
		\emph{Theory Probab. Appl.} \textbf{27} (1982), no.~2, 256--268.
		
		\bibitem{Sturm}
		K.-T. Sturm,
		Analysis on local Dirichlet spaces II. Upper Gaussian estimates for
		the fundamental solutions of parabolic equations,
		\emph{Osaka J. Math.} \textbf{32} (1995), no.~2, 275--312.
		
		\bibitem{Varadhan}
		S.~R.~S. Varadhan,
		On the behavior of the fundamental solution of the heat equation
		with variable coefficients,
		\emph{Comm. Pure Appl. Math.} \textbf{20} (1967), 431--455.
		
		\bibitem{Varadhan2}
		S.~R.~S. Varadhan,
		Diffusion processes in a small time interval,
		\emph{Comm. Pure Appl. Math.} \textbf{20} (1967), 659--685.
		
		\bibitem{WYZ}
		Y. Wang, S. Yang and T. Zhang,
		Small-time annealed large deviations principle for one-dimensional
		diffusions in a random environment,
		arXiv preprint (2026).
		\href{https://arxiv.org/abs/2608.26834}{arXiv:2608.26834}.
		
		\bibitem{PitmanYor}
		J. Pitman and M. Yor,
		Decomposition at the maximum for excursions and bridges of
		one-dimensional diffusions,
		in N. Ikeda, S. Watanabe, M. Fukushima and H. Kunita (eds.),
		\emph{It\^o's Stochastic Calculus and Probability Theory},
		Springer, 1996, pp.~293--310.
		\href{https://www.stat.berkeley.edu/~pitman/agreement.pdf}
		{Authors' reprint}.
		
	\end{thebibliography}
\end{document}